\documentclass[1p,number,a4paper]{elsarticle}
\usepackage{stix}
\usepackage{amsfonts}
\usepackage{amsmath}
\usepackage{amsthm}
\usepackage{mathrsfs}
\usepackage{enumerate}
\usepackage{graphics}
\usepackage{mathtools}
\usepackage{microtype}
\usepackage[defaultlines=4]{nowidow}
\usepackage{xcolor}
\usepackage[czech, english]{babel}
\usepackage{esvect}
\usepackage{complexity}
\usepackage[many]{tcolorbox}
\usepackage{thm-restate}
\usepackage[all]{xy}
\usepackage{cleveref}
\usepackage{relsize}
\usepackage{mathtools}
\usepackage{todonotes}
\usepackage{url}
\theoremstyle{remark}

\newcommand{\cqed}{\ensuremath{\lhd}}

\makeatletter

\tcolorboxenvironment{theorem}{
	colframe=cyan,interior hidden, breakable,before skip=10pt,after skip=10pt,	outer arc=0pt,
	arc=0pt,
	colback=white,
	rightrule=0pt,
	toprule=0pt,
	top=0pt,
	right=0pt,
	bottom=0pt,
	bottomrule=0pt,
}
\tcolorboxenvironment{definition}{
	colframe=olive,interior hidden, breakable,before skip=10pt,after skip=10pt,	outer arc=0pt,
	arc=0pt,
	colback=white,
	rightrule=0pt,
	toprule=0pt,
	top=0pt,
	right=0pt,
	bottom=0pt,
	bottomrule=0pt,
}
\tcolorboxenvironment{fact}{
	colframe=blue, interior hidden, breakable,before skip=10pt,after skip=10pt,	outer arc=0pt,
	arc=0pt,
	colback=white,
	rightrule=0pt,
	leftrule=0.5pt,
	toprule=0pt,
	top=0pt,
	right=0pt,
	bottom=0pt,
	bottomrule=0pt,
}
\tcolorboxenvironment{lemma}{
	colframe=blue,interior hidden, breakable,before skip=10pt,after skip=10pt,	outer arc=0pt,
	arc=0pt,
	colback=white,
	leftrule=0.5pt,
	rightrule=0pt,
	toprule=0pt,
	top=0pt,
	right=0pt,
	bottom=0pt,
	bottomrule=0pt,
}
\tcolorboxenvironment{corollary}{
	colframe=blue,interior hidden, breakable,before skip=10pt,after skip=10pt,	outer arc=0pt,
	arc=0pt,
	colback=white,
	rightrule=0pt,
	toprule=0pt,
	top=0pt,
	right=0pt,
	bottom=0pt,
	bottomrule=0pt,
}

\makeatother

\newcommand{\MSO}{\ComplexityFont{MSO}}
\newcommand{\CMSO}{\ComplexityFont{CMSO}}

\newcommand{\Cc}{\mathscr{C}}
\newcommand{\Dd}{\mathscr{D}}
\newcommand{\Ee}{\mathscr{E}}
\newcommand{\Ff}{\mathscr{F}}
\newcommand{\Gg}{\mathscr{G}}
\newcommand{\Tt}{\mathscr{T}}
\newcommand{\Pp}{\mathscr{P}}

\newcommand{\Rr}{\mathscr{R}}
\newcommand{\Ll}{\mathcal{L}}

\DeclareMathOperator{\ind}{\text{index}}
\DeclareMathOperator{\bind}{\text{b-index}}
\DeclareMathOperator{\sind}{\text{sindex}}

\newcommand{\Trs}{\mathsf T}

\newcommand{\N}{\mathbb{N}}

\newcommand{\ov}[1]{\overline{#1}}
\newcommand{\ti}[1]{\widetilde{#1}}

\newcounter{dummyc}
\newcommand{\gtd}{Gy\'arf\'as decomposition}
\crefname{fact}{Fact}{Facts}
\crefname{conjecture}{Conjecture}{Conjectures}

\newcommand{\ERCagreement}{This paper is part of projects that have received funding from the European Research Council (ERC) under the European Union's Horizon 2020 research and innovation programme (grant agreements No 810115 -- {\sc Dynasnet} and 948057 -- {\sc BOBR}) and from the German
	Research Foundation (DFG) with grant agreement
	No 444419611.\\
	\includegraphics[width=.25\textwidth]{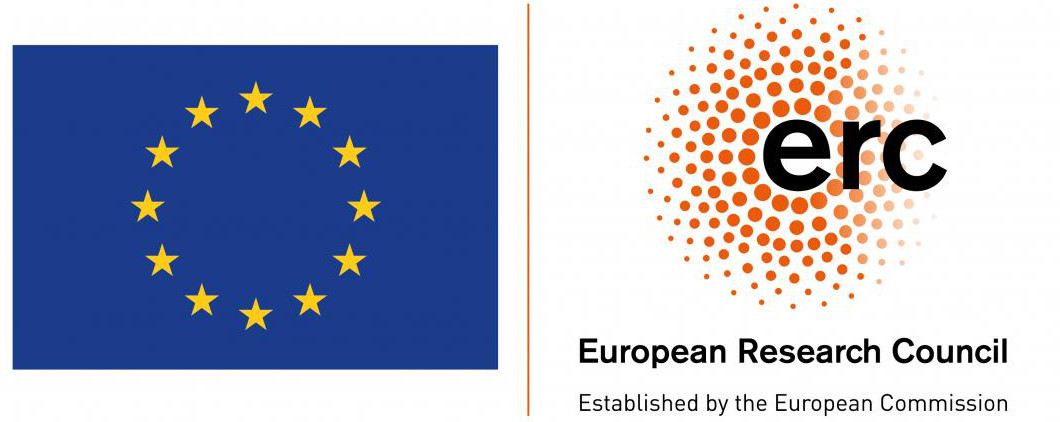}\ 
	\includegraphics[width=.2\textwidth]{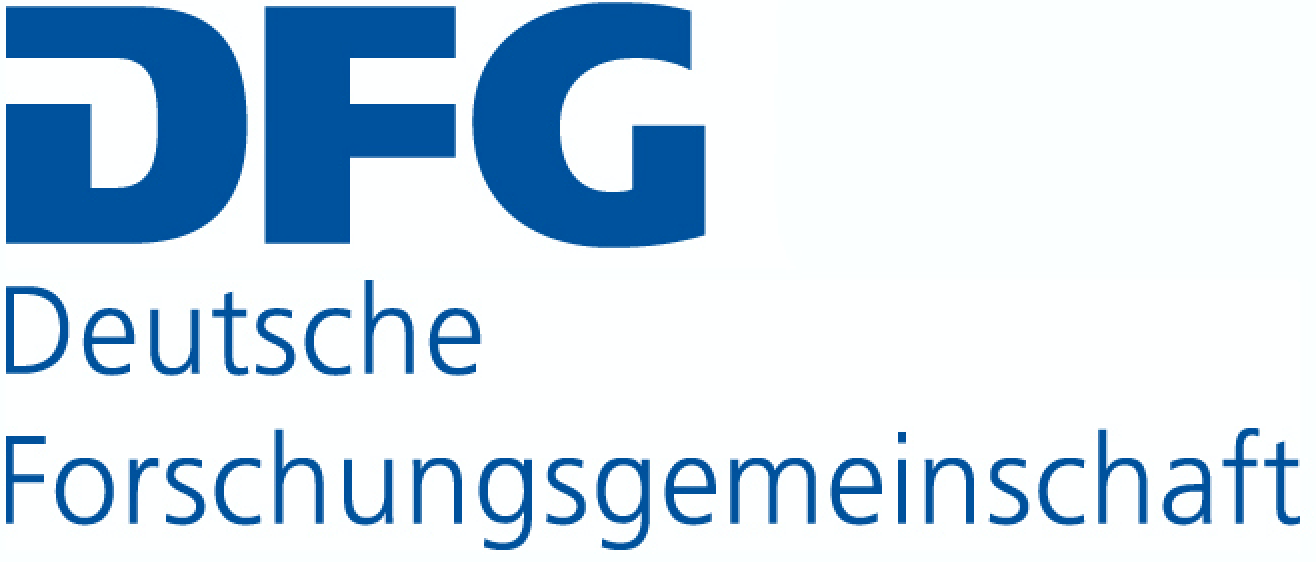}}

\newtheorem{theorem}{Theorem}[section]

\newtheorem{corollary}{Corollary}[section]
\newtheorem{definition}{Definition}[section]
\newtheorem{lemma}{Lemma}[section]

\newtheorem{problem}{Problem}

\newtheorem{conjecture}[problem]{Conjecture}

\newtheorem{fact}[theorem]{Fact}
\crefname{figure}{Figure}{Figures}
\crefname{theorem}{Theorem}{Theorems}
\crefname{definition}{Definition}{Definitions}
\crefname{ext_theorem}{Theorem}{Theorems}
\crefname{corollary}{Corollary}{Corollaries}
\crefname{lemma}{Lemma}{Lemmas}
\crefname{section}{Section}{Sections}

\journal{arXiv}

\begin{document}
	\begin{frontmatter}
	\title{Transducing paths in graph classes with unbounded shrubdepth}
	\tnotetext[ERC]{\ERCagreement}
	\author{Patrice Ossona de Mendez}\address{Centre d'Analyse et de Math\'ematiques Sociales (CNRS, UMR 8557), Paris, France and Computer Science Institute of Charles University, Praha, Czech Republic}\ead{pom@ehess.fr}
	\author{Micha\l{} Pilipczuk}\address{University of Warsaw, Poland}\ead{michal.pilipczuk@mimuw.edu.pl}
	\author{Sebastian Siebertz}\address{University of Bremen, Germany}\ead{siebertz@uni-bremen.de}

	\begin{keyword}
		 shrubdepth, first-order transduction, $\chi$-boundedness
    \end{keyword}
	\begin{abstract}
Transductions are a general formalism for expressing transformations of graphs (and more generally, of relational structures) in logic.
We prove that a graph class $\Cc$ can be $\FO$-transduced from a class of bounded-height trees (that is, has {\em{bounded shrubdepth}}) if, and only if, from $\Cc$ one cannot $\FO$-transduce the class of all paths. 
This establishes one of the three remaining open questions posed by Blumensath and Courcelle about the $\MSO$-transduction quasi-order, even in the stronger form that concerns $\FO$-transductions instead of $\MSO$-transductions.

The backbone of our proof is a graph-theoretic statement that says the following: If a graph $G$ excludes a path, the bipartite complement of a path, and a half-graph as semi-induced subgraphs, then the vertex set of $G$ can be partitioned into a bounded number of parts so that every part induces a cograph of bounded height, and every pair of parts semi-induce a bi-cograph of bounded height. This statement may be of independent interest; for instance, it implies that the graphs in question form a class that is linearly $\chi$-bounded.
\end{abstract}

%
%

\end{frontmatter}

\section{Introduction}\label{sec:intro}

\subsection{Background and the main result}

\paragraph*{Transductions} Transductions provide a model theoretical framework for transformations of relational structures definable in logic. In this work we consider only graphs; these will be sometimes {\em{colored}}, by which we mean that there can be some unary predicates distinguishing subsets of vertices (not necessarily disjoint).
%

To define transductions in this setting, we first introduce interpretations. For a logic $\Ll$, an \emph{$\Ll$-interpretation} $\mathsf I$ consists of two $\Ll$-formulas $\varphi(x,y)$ and $\psi(x)$. Applying $\mathsf I$ to a colored graph $G$ yields a new graph $\mathsf I(G)$, whose vertex set consists of all vertices of $G$ satisfying $\psi$ and whose edge set consists of all pairs of vertices of $G$ satisfying $\varphi$. An \emph{$\Ll$-transduction} $\Trs$ is a mechanism of transforming graphs that consists of an $\Ll$-interpretation $\mathsf I$. The semantic is as follows: for a (colored) graph $G$, $\Trs(G)$ consists of all graphs that can be obtained from $G$ by first {\em{coloring}} $G$ --- extending $G$ by the unary predicates used by $\mathsf I$ and not present in $G$ --- in an arbitrary way  and then applying $\mathsf I$. Note that due to the coloring step, $\Trs(G)$ is a set of graphs rather than a single graph: it contains one output per each possible coloring of the input. (See \cref{sec:prelims} for a formal definition. In particular, one usually allows also the operation of copying vertices, which is immaterial here.)

 
For a graph class $\Cc$, we define $\Trs(\Cc)=\bigcup_{G\in \Cc} \Trs(G)$. 
We say that a class $\Cc$ is {\em{$\Ll$-transducible}} from $\Dd$ if there is an $\Ll$-transduction $\Trs$ such that $\Cc\subseteq \Trs(\Dd)$;  we denote this by $\Cc\sqsubseteq_{\Ll} \Dd$. For the standard logics, $\Ll$-transductions are closed under composition. Hence in this case $\sqsubseteq_{\Ll}$ is a quasi-order on graph classes; we call it the {\em{$\Ll$-transduction quasi-order}}. If classes $\Cc$ and $\Dd$ are such that $\Cc\sqsubseteq_\Ll \Dd$ and $\Dd\sqsubseteq_\Ll \Cc$, then we say that $\Cc$ and $\Dd$ are {\em{$\Ll$-equivalent}}.

Depending on the expressive power of $\Ll$, $\Ll$-transductions provide containment notions of varying strength for graphs and for graph classes. In this work we consider $\Ll\in \{\FO,\MSO,\CMSO\}$, where $\FO$ is the standard first-order logic on graphs, $\MSO$ is the {\em{monadic second order logic}} that extends $\FO$ by the possibility of quantifying over vertex subsets, and $\CMSO$ is an extension of $\MSO$ by allowing counting modular predicates that can be applied to the cardinalities of sets. 


\paragraph{Shrubdepth}
To make the notion of a transduction more concrete, let us consider the following example. A {\em{tree model}} of a graph $G$ is a labelled rooted tree $T$ such that:
\begin{itemize}
	\item the vertices of $T$ are labelled with a finite set of labels;
	\item the leaf set of $T$ is equal to the vertex set of $G$; and
	\item for every pair of vertices $u,v$ of $G$, whether $u$ and $v$ are adjacent in $G$ depends only on the triple of labels: of $u$, of $v$, and of the lowest common ancestor of $u$ and $v$ in $T$.
\end{itemize}
(See~\cref{sec:prelims} for a formal definition.) Observe that if $T$ is a tree model of $G$ of height $h$ and using a label set of size $k$, then $G$ can be transduced from $T$ using a fixed $\FO$-transduction (that depends on $h$ and $k$). The transduction first introduces a coloring of $T$ that distinguishes vertices with different labels, as well as the leaves and the root. Then, the edge relation of $G$ can be recovered by an $\FO$ formula that for given two leaves $u,v$, finds the lowest common ancestor of $u$ and $v$ and compares the relevant triple of labels. Finally, the vertex set is restricted to the leaves of $T$.

Tree models were introduced by Ganian et al.~\cite{Ganian2012,Ganian2017} to define a class parameter called {\em{shrubdepth}}. Precisely, the {\em{shrubdepth}} of a graph class $\Cc$ is the least $h\in \N$ for which the following holds: there is $k\in \N$ such that every graph of $\Cc$ admits a tree model of height at most $h$ that uses a set of at most $k$ labels. The argument from the previous paragraph shows that if $\Cc$ has bounded shrubdepth, then $\Cc$ can be $\FO$-transduced from a class of trees of bounded height. As proved in~\cite{Ganian2017}, the converse is also true in a very strong sense: every class $\FO$-transducible from a class of bounded-height trees in fact has bounded shrubdepth, and this holds even for the stronger notions of $\MSO$- and $\CMSO$-transductions. Thus, the tree models presented above provide a ``canonical form'' for \FO-, \MSO-, and even   \CMSO-transductions from classes of bounded height trees.

Classes of bounded shrubdepth appear to be an interesting concept on its own. As argued in~\cite{Ganian2017}, shrubdepth is an analogue of treedepth suited for the treatment of dense graphs. 
Classes of bounded treedepth have bounded shrubdepth, classes of bounded shrubdepth that are additionally sparse (say, exclude some fixed biclique as a subgraph)  have bounded treedepth and, just as graphs with bounded treedepth are building blocks of graphs in bounded expansion classes \cite{Sparsity}, graphs with bounded shrubdepth are building blocks of graphs in $\FO$-transductions of bounded expansion classes \cite{SBE_TOCL}.
 Also, shrubdepth is functionally equivalent to several other parameters such as {\em{rankdepth}} and {\em{SC-depth}}, in the sense that the boundedness of one is equivalent to the boundedness of the others. It is noteworthy that any hereditary class with bounded shrubdepth can be defined by the exclusion of a finite number of induced subgraphs~\cite{Ganian2017}.

%
%

\paragraph*{$\MSO$-transductions}
The $\MSO$-transduction quasi-order has been studied by Blumensath and Courcelle in~\cite{blumensath10}, who described the following structure. For $n\in \N$, let $\Ff_n$ be the class of rooted forests of height $n$; here, the height of a forest is the length of the longest root-to-leaf path, hence~$\Ff_0$ is the class of edgeless graphs. Further, let $\Pp$~be the class of all paths, $\Tt$ be the class of
all trees, and $\Gg$ be the class of all graphs. Clearly, we have 
$$\emptyset \sqsubseteq_\MSO \Ff_0\sqsubseteq_\MSO \Ff_1\sqsubseteq_\MSO \ldots \sqsubseteq_\MSO \Ff_n\sqsubseteq_\MSO\ldots\sqsubseteq_\MSO\Pp\sqsubseteq_\MSO\Tt\sqsubseteq_\MSO\Gg.$$

Up to $\MSO$-transduction equivalence, these classes correspond to the classes with bounded shrubdepth (transducible from $\Ff_n$ for any $n\in \N$), classes with bounded linear cliquewidth (transducible from $\Pp$), and classes with bounded cliquewidth (transducible from  $\Tt$).
In \cite{blumensath10}, Blumensath and Courcelle proved that this hierarchy is strict. They conjectured that it is also complete in the following sense: every graph class is $\MSO$-equivalent to one of the classes in the hierarchy; see~\cite[Open Problem~9.3]{blumensath10}. 

This conjecture was confirmed for the initial prefix of the hierarchy by Ganian et al.~\cite{Ganian2017} in the following sense: if $\Cc\sqsubseteq_\MSO \Ff_m$ for some $m\in \N$ (equivalently, $\Cc$ has bounded shrubdepth), then in fact $\Cc$ is $\MSO$-equivalent to $\Ff_n$ for some $n\in \N$. 
Therefore, to resolve the open problem of Blumensath and Courcelle it remains to prove the following three conjectures.

\begin{conjecture}	\label{conj:MSO}
	A class $\Cc$ of graphs has bounded shrubdepth if, and only if, 
	the class of all paths is not  $\MSO$-transducible from $\Cc$.
\end{conjecture}

\begin{conjecture}	\label{conj:MSOlcw}
	A class $\Cc$ of graphs has bounded linear cliquewidth if, and only if, the class of all trees is not $\MSO$-transducible from $\Cc$.
\end{conjecture}

\begin{conjecture}	\label{conj:MSOcw}	
	A class $\Cc$ of graphs has bounded cliquewidth if, and only if, the class of all grids is not $\MSO$-transducible from $\Cc$.
\end{conjecture}

We note that a positive resolution of~\cref{conj:MSOcw} would establish another long-standing open problem, {\em{Seese's Conjecture}}.

\cref{conj:MSO,conj:MSOcw} have been established in weaker forms, where one replaces $\MSO$-transductions with $\CMSO$-transductions. More precisely, Kwon et al.~\cite{kwon2021obstructions} proved that every graph class of unbounded shrubdepth contains all paths as vertex-minors, and Courcelle and Oum~\cite{courcelle2007vertex} proved that every graph class of unbounded cliquewidth contains all grids as vertex-minors. Since the vertex-minor relation can be expressed by a $\CMSO$-transduction~\cite{courcelle2007vertex}, this yields right-to-left implications of the $\CMSO$-variants of \cref{conj:MSO,conj:MSOcw}. The left-to-right implications are well-known to hold.

\paragraph*{$\FO$-transductions}
Compared to the above, much less is known about $\FO$-transductions and the picture appears to be much more complicated; see~\cite{arboretum}. However, in recent years there has been an attempt of constructing a sound structural theory for well-structured dense graphs based on this notion; see e.g.~\cite{tww4,tww1,GajarskyHOLR20,SBE_TOCL,StableTWW,msrw,NesetrilMPRS21}. In particular, the following conjectures have been formulated about characterizations of shrubdepth, linear cliquewidth, and cliquewidth in terms of $\FO$-transductions. These are strengthenings of \cref{conj:MSO,conj:MSOlcw,conj:MSOcw}, respectively.

\begin{conjecture}[\cite{arboretum}, repeated in \cite{StableTWW}]
	\label{conj:FO}
	A class $\Cc$ of graphs has bounded shrubdepth if, and only if, 
	the class of all paths is not \FO-transducible from $\Cc$.
\end{conjecture}

\begin{conjecture}[\cite{StableTWW}]
	\label{conj:FOlcw}
	A class $\Cc$ of graphs has bounded linear cliquewidth if, and only if, 
	no class containing some subdivision of every binary tree is
 \FO-transducible from $\Cc$.
\end{conjecture}

\begin{conjecture}[\cite{StableTWW}]
	\label{conj:FOcw}	
	A class $\Cc$ of graphs has bounded cliquewidth if, and only if, 
	no class containing some subdivision of every wall is \FO-transducible from $\Cc$.
\end{conjecture}

In this work we prove \cref{conj:FO}. As shown in~\cite{Ganian2017}, the class of all paths does not have bounded shrubdepth, hence it is not \FO-transducible (equivalently, \MSO- or \CMSO-transducible) from any class with bounded shrubdepth (equivalently, from any class of bounded height trees). Therefore, \cref{conj:FO} follows from the next theorem, which is the main result of our paper.

\begin{theorem}\label{thm:main-shb}
	If the class of all paths is not \FO-transducible from a class of graphs $\Cc$, then $\Cc$ has bounded shrubdepth.
\end{theorem}

Since every $\FO$-transduction is also an $\MSO$-transduction, 
this also confirms \cref{conj:MSO}. More generally, together with the results of~\cite{Ganian2017}, \cref{thm:main-shb} proves that the property of having bounded shrubdepth is the largest property of graph classes that is closed under taking $\FO$-transductions (is a so-called {\em{$\FO$-transduction ideal}}) and whose restriction to classes that exclude a fixed biclique as a subgraph is the property of having bounded treedepth.

\subsection{Graph-theoretic statements}

The main idea behind the proof of \cref{thm:main-shb} is to analyze graphs that exclude some simple substructures from which long paths can  easily be transduced. These are: paths, complements of paths, and half-graphs. We prove two purely graph-theoretic statements which explain that graphs excluding the substructures mentioned above can be nicely decomposed into cographs. We need a few definitions to state the decomposition theorems formally.

For a graph $G$ and disjoint subsets of vertices $A,B\subseteq V(G)$, by $G[A]$ we denote the subgraph induced by $A$ in $G$, and by $G[A,B]$ we denote the bipartite graph {\em{semi-induced}} by $A$ and $B$ in $G$. The latter is the graph with vertex set $A\cup B$ and edge set comprising of all edges of $G$ with one endpoint in $A$ and second in $B$. By $P_t$, $\ov{P}_t$, and $\ti{P}_t$ we respectively denote the path on $t$ vertices, its complement, and its {\em{bipartite complement}}: the graph obtained from $P_t$ by taking its unique bipartition $A,B$ and complementing the edge relation on the set $A\times B$. We also use the order-$k$ {\em{half-graph}} $H_k$, and the order-$k$ {\em{universal threshold graph}} $R_k$; see~\cref{fig:HkRk}.

\begin{figure}[h!t]
	\centering
	\includegraphics[width=.5\linewidth]{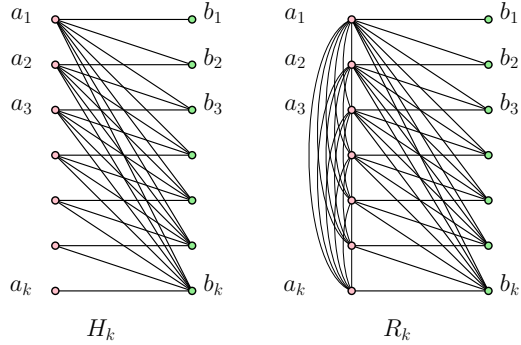}
	\caption{Left: the half-graph $H_k$. Right: the universal threshold graph $R_k$.}
	\label{fig:HkRk}
\end{figure}

For decompositions we rely on the well-known class of {\em{cographs}} (that is, $P_4$-free graphs) and on their bipartite counterparts, {\em{bi-cographs}} (see \cref{sec:prelims} for a definition).
For a graph $G$, a {\em{cosplit}} of $G$ is a partition $\Pp$ of the vertex set of $G$ such that for every part $A\in \Pp$, the induced subgraph $G[A]$ is a cograph. A cosplit $\Pp$ is moreover a {\em{$2$-cosplit}} if for every pair of distinct parts $A,B\in \Pp$, the semi-induced bipartite graph $G[A,B]$ is a bi-cograph. The {\em{height}} of a cosplit $\Pp$ is the maximum height of the cographs in $\{G[A]\colon A\in \Pp\}$, where the height of a cograph is the height of its cotree. We define the height of a $2$-cosplit analogously, but we also take into account the heights of bi-cographs in $\{G[A,B]\colon A,B\in \Pp,\ A\neq B\}$. Note that the height of a $2$-cosplit $\Pp$ may be larger than the height of $\Pp$ treated as a (standard) cosplit; this distinction will be always clear from the context. The {\em{size}} of a cosplit $\Pp$ is simply $|\Pp|$.
The minimum size of a cosplit of a graph $G$ is its
\emph{$c$-chromatic number} (also called \emph{$P_4$-chromatic number})~\cite{Gimbel20103437}.

Our main decomposition results can be now stated as follows.

\begin{theorem}\label{thm:induced}
For every pair of integers $t,k\in \N$ there exists $N\in \N$  such that the following holds: Every graph that excludes $P_t$, $\ov{P}_t$, and $R_k$ as induced subgraphs admits a cosplit of size at most $N$ and height at most $4k$.
\end{theorem}

\begin{theorem}\label{thm:semi-induced}
	For every pair of integers $t,k\in \N$ there exists $N\in \N$ 
	such that the following holds: Every graph that excludes $P_t$, $\ti{P}_t$, and $H_k$ as semi-induced subgraphs admits a $2$-cosplit of size at most $N$  and height at most $4k$. 
\end{theorem}

Thus, \cref{thm:semi-induced} provides a much stronger form of a decomposition --- a $2$-cosplit instead of a cosplit --- at the expense of a stronger assumption about the excluded substructures. We note that our proofs yield bounds $N\leq 4\cdot (2t-5)^{k-1}$ in case of \cref{thm:induced} (for $t\geq 4$) and $N\leq 3^{2^k\cdot t^{3k-2}}$ in case of \cref{thm:semi-induced} (for $t\geq 5$). In the proof of \cref{thm:main-shb} we use only \cref{thm:semi-induced}, but \cref{thm:induced} is actually a stepping stone in the proof of \cref{thm:semi-induced} and provides some interesting corollaries on its own; we will discuss these later.

The proofs of \cref{thm:induced,thm:semi-induced} rely on the influential Gy\'arf\'as' path method~\cite{gyarfas1987problems}. Through a suitable understanding, this method provides a convenient decomposition notion for $P_t$-free graphs, which we call a {\em{\gtd}}. We sometimes consider a \gtd{} of the graph itself and sometimes of its complement; this is why we assume that the graph excludes both $P_t$ and its (bipartite) complement. The main idea is to apply induction where the measure of progress is (roughly) the largest $k$ such that $R_k$ or $H_k$ can be found as a (semi-)induced subgraph. This choice was inspired by the work of Gajarsk\'y et al.~\cite{StableTWW}. 

The proof of \cref{thm:main-shb} proceeds roughly as follows. The left-to-right implication is known~\cite{Ganian2017}, hence it remains to show that if from a class of graphs $\Cc$ one cannot transduce the class of all paths, then $\Cc$ has bounded shrubdepth. The assumption implies that there are some $t,k\in \N$ such that all graphs in $\Cc$ exclude $P_t$, $\ti{P}_t$, and $H_k$ as semi-induced subgraphs. By \cref{thm:semi-induced}, every graph $G\in \Cc$ admits a $2$-cosplit of bounded size and height. This allows us to {\em{sparsify}} $G$, that is, find a sparse graph $G'$ that encodes $G$ in the following sense: both $G'$ can be transduced from $G$ and $G$ can be transduced back from $G'$. Letting $\Dd$ be the class comprised of all graphs $G'$ as above, we still have that from $\Dd$ one cannot transduce the class of all paths, but now moreover we have that $\Dd$ is sparse (formally, it is degenerate). Using known connections between treedepth and the existence of long paths one can now argue that $\Dd$ actually has bounded treedepth. So $\Cc$, being transducible from $\Dd$, has bounded shrubdepth.

\subsection{Other corollaries}

Finally, we discuss several other statements that can be inferred from \cref{thm:induced,thm:semi-induced}. The first one concerns $\chi$-boundedness of the considered graph classes. Recall that a class of graphs $\Cc$ is {\em{$\chi$-bounded}} if there is a function $f\colon \N\to \N$ such that $\chi(G)\leq f(\omega(G))$ for every graph $G\in \Cc$, where $\chi(G)$ and $\omega(G)$ respectively denote the chromatic number and the clique number of $G$. If $f$ can be additionally chosen to be a linear function, then $\Cc$ is {\em{linearly $\chi$-bounded}}. It is known that for every $t\in \N$, graphs excluding $P_t$ as an induced subgraph are $\chi$-bounded~\cite{gyarfas1987problems}. We show that if one additionally excludes the complement of a path and a universal threshold graph, then the resulting class is even linearly $\chi$-bounded.

\begin{corollary}\label{cor:chi}
 For every pair of integers $t,k\in \N$, the class of graphs that exclude $P_t$, $\ov{P}_t$, and~$R_k$ as induced subgraphs is linearly $\chi$-bounded.
\end{corollary}
\begin{proof}
	By \cref{thm:induced}, there exists an integer $N$ such that every graph  $G$ that excludes $P_t$, $\ov{P}_t$, and $R_k$ as induced subgraphs has a cosplit $\Pp$ of size at most $N$. As cographs are perfect we have
 \[
 \chi(G)\leq\sum_{X\in\Pp}\chi(G[X])=\sum_{X\in\Pp}\omega(G[X])\leq N\omega(G).\qedhere
 \]
\end{proof}

We remark that in \cref{cor:chi}, excluding a path and an antipath is not sufficient to guarantee linear $\chi$-boundedness, for the following reason.
The lexicographic product of two $P_t$-free graphs is $P_t$-free and, as the complement of the lexicographic product of two graphs is (isomorphic to) the product of their complements, the same holds for $\ov{P}_t$-free graphs.
So, consider a triangle-free graph~$G_t$ with $t$ vertices and fractional chromatic number $\chi_f(G_t)\geq \frac{1}{9}\sqrt{t/\log t}$ (see \cite{R3t}).
Obviously, $G_t$ is $P_{t+1}$-free. The lexicographic powers of $G_t$ are $P_{t+1}$ and $\overline{P}_{t+1}$-free and have chromatic number $\chi(G)\geq \omega(G)^{\log_2\chi_f(G_t)}$ (see, for example, \cite{msrw}). Hence, the exponent of $\omega(G)$ has to grow at least as $(1-o(1))\,\log_2 t$.
For instance,  the lexicographic powers of $C_5$ exclude both~$P_5$ and $\overline{P}_5$ as induced subgraph and satisfy $\chi(G)\geq\omega(G)^c$, where $c=\log_2\chi_f(C_5)=\log_2 5-1\approx 1.32$.
\medskip

The next corollary concerns the Erd\H os-Hajnal property of graphs excluding a universal threshold graph. 
Recall that a graph class $\Cc$ has the {\em{Erd\H os-Hajnal property}} if there exists $\delta>0$ such that every $n$-vertex graph $G\in \Cc$ contains a {\em{homogeneous set}} --- a clique or an independent set --- of size at least $n^{\delta}$. It can be derived from the stable regularity lemma of Malliaris and Shelah~\cite{Malliaris2014} that for every $k\in \N$, graphs excluding $R_k$ as an induced subgraph have the Erd\H os-Hajnal property. See also~\cite{chernikov2015}, and~\cite[Theorem 2.8 in Chapter~5]{notes} for a streamlined presentation\footnote{The statement presented in~\cite{notes} assumes excluding a half-graph as a semi-induced subgraph, but the proof actually works also for excluding a universal threshold graph as an induced subgraph.} yielding $\delta=\frac{1}{2k+2}$. We now show that from \cref{thm:induced} one can infer essentially the same result.

%

\begin{corollary}
Every graph $G$ on $n$ vertices that does not contain $R_k$ as an induced subgraph contains a homogeneous set of size at least $n^{1/2k}/4$.
\end{corollary}
\begin{proof}
    As shown in the proof, for $t\geq 4$, the value of $N$ in~\cref{thm:induced} can be upper bounded by $4\cdot (2t-5)^{2k-2}$.
	Let $h$ be the maximum size of a homogeneous subset in $G$; we may assume that $h\geq 2$. Obviously, $G$ excludes $P_{2h+1}$ and $\ov{P}_{2h+1}$ as induced subgraphs, for taking every second vertex of such an induced subgraph would yield a homogeneous set of size $h+1$.
	Thus, by~\cref{thm:induced}, $G$ admits a cosplit of size at most $p=4\cdot (4h-3)^{2k-2}$. Hence, $G$ contains an induced cograph with at least $n/p$ vertices. It is well-known that every cograph on $m$ vertices contains a homogeneous set of size at least $\sqrt{m}$, hence $G$ contains a homogeneous set of size at least $\sqrt{n/p}$. It follows that 
	$n/p\leq h^2$, that is, $n\leq 4\,(4h-3)^{2k-2}\,h^2<(4h)^{2k}$. Hence, 
	$h>n^{1/2k}/4$.
\end{proof}



\medskip

Here is another corollary, this time of \cref{thm:semi-induced}. It is easy to see that for every fixed $h\in \N$, cographs of height at most $h$ and bi-cographs of height at most $h$ form classes that are set-defined. Here, a class is \emph{set-defined} if it comprises of induced subgraphs of a single infinite graph that is $\FO$-interpretable in the countable pure set (a set with no relations); see~\cite{disc_arxiv} for a wider discussion. From \cref{thm:semi-induced} it then follows that for every fixed $t,k\in \N$, graphs excluding $P_t$, $\ti{P}_t$ and $H_k$ as semi-induced subgraphs are also set-defined. 
Applying the results of \cite{disc_arxiv}, we conclude that these graphs have {\em{bounded discrepancy}} in the following sense.
 
 \begin{corollary}
 For every pair of integers $t,k\in \N$ there exists a constant $C$ such that every graph $G$ that excludes $P_t$, $\ti{P}_t$ and $H_k$ as semi-induced subgraphs can be $2$-colored in such a way that for every vertex $u$ of $G$, the difference between the numbers of neighbors of $u$ in each of the colors is at most $C$.  
 \end{corollary}

\section{Preliminaries}\label{sec:prelims}

\paragraph*{Graph notation} We use standard notation for graphs. Whenever speaking about bipartite graphs, we always assume that the graph is given together with a fixed {\em{bipartition}}: a partition of the vertex set into two independent sets, called {\em{sides}}. Note that the bipartition is unique if the bipartite graph is connected. 

The {\em{complement}} of a graph $G$, denoted $\ov{G}$, is the graph on the same vertex set where two vertices are adjacent if and only if they are non-adjacent in $G$. The {\em{bipartite complement}} of a bipartite graph $G$, denoted $\ti{G}$, the bipartite graph on the same vertex set as $G$ and with the same bipartition, where two vertices from different sides of the bipartition are adjacent if and only if they were non-adjacent in $G$. Note that the bipartite complement of a bipartite graph may differ from its (standard) complement when treated as a (non-bipartite) graph. This distinction will be always clear from the context.

By $P_t$ we denote the path on $t$ vertices. Consequently, $\ov{P_t}$ and $\ti{P_t}$ are respectively the complement and the bipartite complement of $P_t$ (taken with respect to the unique bipartition). We also consider the {\em{half-graph}} of order $k$, denoted $H_k$, and the {\em{universal threshold graph}} of order $k$, denoted~$R_k$. These are depicted in \cref{fig:HkRk} and defined as follows. In both cases, the vertex set consists of vertices $a_1,\dots,a_k,b_1,\dots,b_k$, there are edges $a_ib_j$ for all $1\leq i\leq j\leq k$, and the vertices $b_1,\ldots,b_k$ form an independent set. The difference is that in $H_k$ the vertices $a_1,\ldots,a_k$ also form an independent set, while in $R_k$ they form a clique. We shall treat $H_k$ as a bipartite graph; the bipartition is $\{\{a_1,\ldots,a_k\},\{b_1,\ldots,b_k\}\}$.

The name of $R_k$ --- the universal threshold graph of order $k$ --- is motivated by the following easy fact. Here, a {\em{threshold graph}} is a graph that can be obtained from the empty graph by iteratively adding a universal vertex or an isolated vertex.

\begin{fact}
	\label{fact:univ_thr}
	Every threshold graph with at most $k$ vertices is an induced subgraph of $R_k$.
\end{fact}
%

For a graph $G$ and a subset of vertices $A$, by $G[A]$ we denote the subgraph of $G$ induced by~$A$. For disjoint $A,B\subseteq V(G)$, by $G[A,B]$ we denote the bipartite subgraph {\em{semi-induced}} by $A$ and~$B$: it is the bipartite graph with bipartition $\{A,B\}$ where $u\in A$ and $v\in B$ are adjacent if and only if $u$ and $v$ are adjacent in $G$. Note that thus, $A$ and $B$ are independent sets in~$G[A,B]$.

\paragraph*{Shrubdepth, cographs and bi-cographs}  
A \emph{rooted tree} is a tree with a distinguished vertex, its \emph{root}. 
The \emph{height} of a rooted tree is the maximum number of edges in a root-to-leaf path. Every rooted tree $T$ defines a partial order $\preceq_T$ on its vertex set: $u\preceq_T v$ if $u$ and $v$ the (unique) path from the root to $v$ contains $u$. If $u\preceq_T v$ we say that $u$ is an \emph{ancestor} of $v$ in $T$. The \emph{least common ancestor} $u\wedge_T v$ of two vertices  is the maximum vertex $w$ (with respect to $\preceq_T$) with $w\preceq_T u$ and $w\preceq_T v$. For a rooted tree $T$, we denote by $L(T)$ the set of \emph{leaves} of $T$, which are the vertices of $T$ that are maximal in 
$\preceq_T$, and by $I(T)$ the set of internal vertices of $T$, which is defined by $I(T)=V(T)\setminus L(T)$.

The following notion of a tree model generalizes the notion of a cotree of a cograph and will allow a unified approach to cographs, bi-cographs, and  shrub-depth.
Let $k,h$ be  positive integers. A \emph{$k$-colored height-$h$ tree model} is a
triple $(T,c,f)$, where $T$ is a rooted tree with height at most $h$, $c\colon L(T)\to [k]$ is a coloring of the leaves of $T$, and $f\colon I(T)\times[k]
\times [k]\to\{0,1\}$ is a function that is symmetric on its two last arguments: $f(x,y,z)=f(x,z,y)$.
A $k$-colored height-$h$ tree $(T,c,f)$
defines a graph $G$ with vertex set $L(T)$ and edge set 
\[
E(G)=\{(u,v)\in V(G)\times V(G)\colon u\neq v\text{ and }f(u\wedge_T  v,c(u),c(v))=1\}.
\]

In this setting, we have:
\begin{itemize}
	\item  A class $\Cc$ has \emph{bounded  shrub-depth} if there exist integers $k,h$ such that every graph in $\Cc$ is defined by a $k$-colored height-$h$ tree model. 
\item A graph is a \emph{cograph} \cite{Corneil1981} (with \emph{height} at most $h$) if it is defined by a $1$-colored (height-$h$) tree model. In this case, the tree model is called a \emph{cotree}.
\item A bipartite graph is a \emph{bi-cograph} \cite{Giakoumakis1997} (with \emph{height} at most $h$) if it is defined by a $2$-colored (height-$h$ tree) model in which the coloring $c$ is  the coloring defined by the bipartition of the vertex set.
(Thus, $f(x,1,1)=f(x,2,2)=0$ for every internal vertex $x$ of the tree model.) In this case, the tree model is called a \emph{bi-cotree}.
\end{itemize}

It is instructive to take a closer look at the combinatorics of cotrees and bi-cotrees. Suppose~$T$ is a cotree of a cograph $G$ and $x$ is an internal vertex of $T$. Then the subtree of $T$ rooted at $x$ defines the subgraph of $G$ induced by the leaves of $T$ that are descendants of $x$. Further, depending on whether $f(x,1,1)=0$ or $f(x,1,1)=1$, this subgraph is either the disjoint union or the join of subgraphs defined by the subtrees rooted at the children of $x$. Here, the {\em{join}} of a collection of graphs is obtained by taking their disjoint union and making every pair of vertices originating from different graphs adjacent. For bi-cotrees we have a similar characterization, except that the cases $f(x,1,2)=f(x,2,1)=0$ and $f(x,1,2)=f(x,2,1)=1$ correspond to taking the bipartite disjoint union or the bipartite join of bipartite graphs, defined analogously.

We also have the following.

\begin{fact}\label{fc:co-bi-co}
A  graph $G$ is a bi-cograph with height $h$ if and only if it is a semi-induced 
subgraph of a cograph with height $h$.
\end{fact}
\begin{proof}
	Assume $G$ is a semi-induced subgraph of a cograph $H$ with height $h$. By considering an induced subgraph if necessary, we may assume $V(G)=V(H)$.
	Let $(T,c,f)$ be a $1$-colored height-$h$ tree model of $H$ and let $A$ and $B$ be the sides of $G$.
	Let $c'\colon V(G)\rightarrow[2]$ be defined by $c'(v)=1$ if $v\in A$ and $c'(v)=2$, otherwise. Let $f'\colon I(T)\times[2]\times[2]\to\{0,1\}$ be defined by setting $f'(x,1,1)=f'(x,2,2)=0$ and $f'(x,1,2)=f'(x,2,1)=f(x,1,1)$, for every internal vertex $x$ of $T$. Then $(T,c',f')$ is a $2$-colored height-$h$ tree model of $H[A,B]$, that is, of $G$.
   
   Conversely, assume $G$ is a bi-cograph with sides $A$ and $B$ and a $2$-colored height-$h$ tree model $(T,c',f')$. Let $c\colon L(T)\to [1]$ be the constant function, and let $f\colon I(T)\times[1]\times[1]\to\{0,1\}$ be defined by setting $f(x,1,1)=f'(x,1,2)$ for every internal vertex $x$ of $T$. Then $(T,c,f)$ is a $1$-colored height-$h$ tree model of a cograph $H$ with $H[A,B]=G$.
\end{proof}

%

\paragraph*{Transductions} 
We start by recalling some general definitions on relational structures. Then, we will focus on (uncolored or colored) graphs.

Recall that a \emph{(relational) signature} is a set $\sigma$ of relation symbols, each with a prescribed arity. 
To a signature $\sigma$ we associate the signature $\sigma^+$ obtained by adding to $\sigma$ a countable set of unary predicates $\{P_i\colon i\in \N\}$.
In this paper we consider only the signature $\sigma_0$ of (uncolored) graphs, which consists of a single relation symbol $E$ with arity $2$, and the signature $\sigma_0^+$, where the added predicates distinguish subsets of vertices. We often call those subsets {\em{colors}}, but mind that we do not assume that they are disjoint: a vertex can be assigned any subset of colors.
Thus, uncolored graphs are $\sigma_0$-structures and colored graphs are $\sigma^+$-structures.

Let $\sigma$ be a signature of relational structures.
A \emph{coloring} of a $\sigma$-structure $\mathbf A$ is any $\sigma^+$-structure obtained from $\mathbf A$ by interpreting the predicates $\{P_i\colon i\in \N\}$ in any way\footnote{This mapping can be arbitrary. In particular, the predicates do not need to be defined by any formula.}.



Let $\sigma,\tau$ be two signatures.
For a $\sigma$-structure $\mathbf A$ and a formula $\varphi$ with $k$ free variables, we define
\[
\varphi(\mathbf A)=\{(v_1,\dots,v_k)\in A^k\colon \mathbf A\models \varphi(v_1,\dots,v_k)\}.
\]
A \emph{simple interpretation} of $\tau$-structures in $\sigma$-structures is a tuple
$\mathsf I=(\rho_R)_{R\in\tau\cup\{0\}}$ of first-order formulas such that $\rho_0$ has a single variable and $\rho_R$ has $k$ free variables if $k$ is the arity of $R$. For a $\sigma$-structure $\mathbf A$, the $\tau$-structure $\mathsf I(\mathbf A)$ has domain
$\rho_0(\mathbf A)$ and is such that for each $R\in\tau$ with arity~$k_R$ we have
$R(\mathbf B)=\rho_r(\mathbf A)\cap \rho_0(\mathbf A)^{k_R}$.

\pagebreak
A \emph{(non copying) transduction}\footnote{A more general notion of a transduction allows the so-called \emph{copying} operation: blowing up every vertex into a bounded size clique. In this paper we prove that the class of all paths can be obtained from any class with unbounded shrubdepth using a non-copying transduction, hence there is no need for this extension.} 
$\Trs$ is defined by 
an interpretation $\mathsf I$ in colored graphs. The semantic is as follows: for an uncolored graph $G$, we set
\[
\Trs(G)=\{\mathsf I\circ\Lambda(G)\colon \Lambda\text{ is a coloring of }G\}.
\]

In other words, one can think of a transduction as of a non-deterministic mechanism that inputs a graph, colors it arbitrarily, and applies a fixed interpretation at the end. Then $\Trs(G)$ is the set of  all graphs that can be obtained in this manner. Note that even though there exist infinitely many colorings of $G$, only a finite numbers of predicates appear in the formulas defining the interpretation $\mathsf I$. Hence, for a graph $G$, the set $\Trs(G)$ is always~finite.

For a class of graphs $\Cc$ we define $\Trs(\Cc)=\bigcup_{G\in\Cc}\Trs(G)$. The terminology above can be lifted to colored graphs (and in fact, to relational structures) in the expected manner; we assume here that all unary predicates introduced in the coloring step are distinct from all unary predicates already used in the graph.

We note that the composition of two transductions is a transduction.
A class of graphs  $\Dd$ can be \emph{transduced} from a class of graphs $\Cc$ if there exists a transduction $\Trs$ such that $\Dd\subseteq \Trs(\Cc)$.
Two classes $\Cc$ and $\Dd$ are called \emph{transduction-equivalent} if $\Dd$ can be transduced from $\Cc$ and $\Cc$ can be transduced from $\Dd$.

We shall consider a very restricted type of interpretations.
\begin{itemize}
	\item 
A formula $\eta(x,y)$ in the signature of (colored) graphs naturally defines a simple interpretation~$\mathsf I_\eta$ of (colored) graphs in (colored) graphs, where $\rho_0$ is a tautology (i.e. a formula that is always satisfied), $\rho_E(x,y)\coloneqq (x\neq y)\wedge(\eta(x,y)\vee\eta(y,x))$, and (if the target structure is colored) $\rho_{P_i}(x)\coloneqq P_i(x)$.
(Note that the formula $\eta$ can make use of the predicates $P_i$ if the source graphs are colored.)

\item 
A \emph{simple equivalence} of two classes $\Cc$ and $\Dd$ of (colored) graphs is a pair $(\varphi,\psi)$ of first-order formulas such that $\mathsf I_\varphi$ is a bijection from $\Cc$ to $\Dd$ with the inverse mapping being $\mathsf I_\psi$. That is, $\mathsf I_\psi\circ \mathsf I_\varphi$ is the identity on $\Cc$.
\end{itemize}

Note that the existence of a simple equivalence between two classes $\Cc$ and $\Dd$ is a very strong form of transduction equivalence.

\section{\gtd}

As explained in \cref{sec:intro}, our main tool will be the Gy\'arf\'as' path argument~\cite{gyarfas1987problems}. For convenience, we encapsulate it in an abstract notion of a {\em{\gtd}}. 
We shall see that this decomposition is very close
to the one used to define treedepth and, more generally, to the
property of Tr\'emaux trees (a.k.a.\ Depth-First Search trees).

\begin{definition}
	\label{def:dec}
	A \emph{\gtd} of a connected graph $G$ is a
	rooted tree $Y$ satisfying the following properties:
	\begin{enumerate}
		\item The nodes of $Y$ are pairwise disjoint and non-empty subsets of $V(G)$, called {\em{bags}}, whose union is equal to $V(G)$. In other words, the node set of $Y$ is a partition of $V(G)$.
		\item The root bag of~$Y$ consists of a single
		vertex, called the \emph{initial vertex} of the {\gtd} $Y$.
		\item\label{enum:DFS} 
		If vertices $u,u'\in V(G)$ are adjacent in $G$ and $B,B'$ are bags of $Y$ such that $u\in B$ and $u'\in B'$, then one of the bags $B,B'$ is an ancestor of the other in $Y$. (Note that possibly $B=B'$, as every node is an ancestor of itself.)
		\item\label{enum:con} For every bag $B$ of $Y$, the subgraph of $G$ induced by the union of $B$ and all its descendants in $Y$ is connected.
		\item\label{enum:hook} For every non-root bag $B$ of $Y$, there exists a
		vertex $h(B)$, called the \emph{hook} of $B$, which belongs to the parent of $B$ in $Y$ and satisfies the following property: $h(B)$
		is adjacent to all the vertices of $B$ and non-adjacent to all the vertices contained in strict descendants of $B$.
	\end{enumerate}
	A {\gtd} of a disconnected graph $G$ is the union of \gtd s of the connected components of $G$.
\end{definition}
%

We first note that every graph has a {\gtd} and that such a decomposition can be efficiently computed. The proof closely follows the construction of Gy\'arf\'as' presented in~\cite{gyarfas1987problems}.

\begin{lemma}
	\label{lem:GTalgo}
	For every connected graph $G$ and vertex $r$ of $G$, there exists a {\gtd} of $G$ whose initial vertex is $r$. Moreover, given $G$ and $r$, such a {\gtd} can be computed in polynomial time.
\end{lemma}	
\begin{proof}
    Consider the following recursive procedure that given a connected graph $G$ and a non-empty subset of vertices $A$,
    constructs a 	
	decomposition of $G$ into a tree of bags with $A$ being the root bag as follows. Consider the connected components $C_1,\dots,C_\ell$ of
	$G-A$. 
	For each $1\leq i\leq \ell$ choose an arbitrary vertex $v_i\in A$
	with a neighbor in $C_i$. Such a vertex must exist as $G$ is connected. 
	For each  $1\leq i\leq \ell$, apply the procedure recursively to $C_i$ (which is connected by definition) and the subset 
	$A_i=N(v_i)\cap V(C_i)\subseteq V(C_i)$. Finally, combine the obtained decompositions by attaching their root bags $A_i$ as children of a new root bag $A$. 
	It is easily checked that if we apply the procedure to $G$ and $A=\{r\}$, then the obtained
	decomposition is a \gtd of~$G$; the vertices $v_i$ serve as hooks. Also, the procedure clearly runs in polynomial time.
\end{proof}

%
%

Let $Y$ be a {\gtd} of a graph $G$. 
The \emph{level} of a bag $B$ in $Y$ is the length of the path in~$Y$ linking $B$ to a root bag.
In particular, every root bag has level $0$. The \emph{height} of a {\gtd} $Y$ is just the height of $Y$ treated as a rooted forest; equivalently, it is the maximum level among the bags of $Y$. The key observation of Gy\'arf\'as is that excluding a path as an induced subgraph gives an upper bound on the height of a {\gtd}.

\begin{lemma}	\label{fact:height_gtd}
	Let $Y$ be a {\gtd} of a $P_t$-free graph $G$. Then the height of $Y$ is at most $t-2$.
\end{lemma}
\begin{proof}
	Let $B_0-B_1-\dots-B_\ell$ be the longest root-to-leaf path in $Y$, where $B_0$ is a root bag. Let~$v_\ell$ be any vertex of $B_{\ell}$, and for $0\leq i< \ell$, let $v_i=h(B_{i+1})$. By \cref{enum:hook} of \Cref{def:dec}, vertices $v_0,v_1,\dots,v_\ell$ induce a path in $G$. Thus $\ell+1<t$, implying that the height of $Y$ is at most~$t-2$.
\end{proof}

{\gtd}s of bipartite graphs will be of prime importance in this paper. In this setting, it is worth noticing the next property  (see \cref{fig:gt}).

\begin{lemma}\label{lem:bip-gtd}
	Let $G$ be a connected bipartite graph, and let $Y$ be a {\gtd} of~$G$. Then every bag of $Y$ is an independent sets in $G$. Moreover, vertices in bags of odd levels belong to one side of the bipartition, and vertices in bags of even levels belong to the other side of the bipartition.
\end{lemma}
\begin{proof}
	The root bag consists of one vertex, hence it is obviously an independent set contained in one side of the bipartition. Further, every other bag is included in the neighborhood of its hook, hence it is an independent set contained in the other side than the hook. The claim follows.
\end{proof}

\begin{figure}[th]
	\centering
	\includegraphics[width=0.8\linewidth]{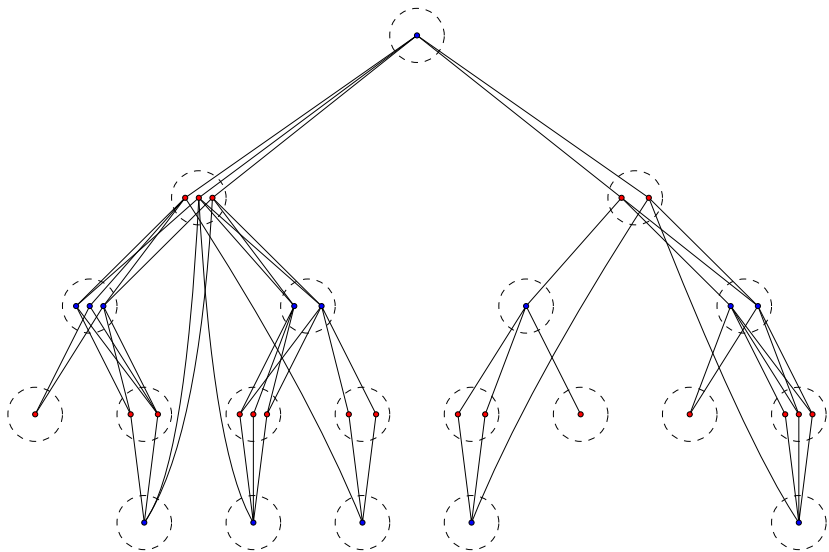}
	\caption{Example of a {\gtd} of a connected bipartite graph. All bags are independent sets.}
	\label{fig:gt}
\end{figure}

\section{Constructing a cosplit: proof of \cref{thm:induced}}


In this section we prove \cref{thm:induced}. The main idea is to apply induction on the maximum $k$ such that $R_k$ is an induced subgraph of the considered graph. As in the induction we will often complement the considered graph, it would be convenient if this parameter would not change under complementation. This is not exactly the case, as $R_k$ is not invariant under complementation. For this technical reason, our induction will use the complementation-invariant notion of a {\em{strong index}} defined below.

\begin{definition}\label{def:sindex}
	The \emph{strong index} of a graph $G$, denoted $\sind(G)$, is the
	maximum integer $k$ such that $G$ contains distinct vertices
	$a_1,\dots,a_k,b_1,\dots,b_k$ such that for all $1\leq i<j\leq k$,
	the vertex $a_i$ is adjacent to the vertex $b_j$, the vertex
	$b_i$ is not adjacent to the vertex $a_j$, the vertices $a_1,\ldots,a_k$ form a clique,
	and the vertices $b_1,\ldots,b_k$ form an independent set (see \cref{fig:sindex}).
\end{definition}

\begin{figure}[h!t]
	\centering
	\includegraphics[width=.5\linewidth]{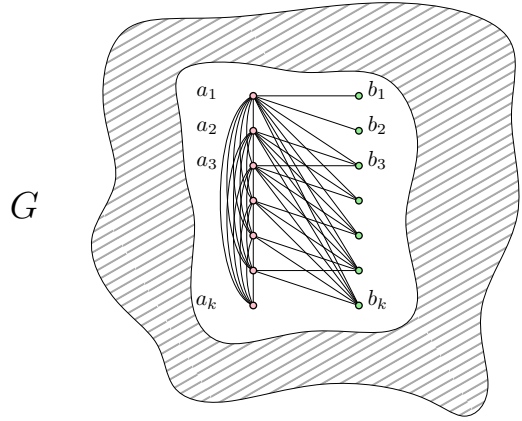}
	\caption{A configuration witnessing $\sind(G)\geq k$.}
	\label{fig:sindex}
\end{figure}

Note that in the above definition, we do not impose any condition on the adjacency between vertices $a_i$ and $b_i$, for $1\leq i\leq k$. For this reason, we have the following simple observation.

\begin{fact}
	\label{fact:sindc}
	For every graph $G$ we have $\sind(\ov{G})=\sind(G)$.
\end{fact}
\begin{proof}
	As $G\mapsto\overline{G}$ is an involution, it suffices to prove $\sind(\ov{G})\geq\sind(G)$.
	Let $a_1,\dots,a_k$, $b_1,\dots,b_k$ witness  $\sind(G)=k$.
	Then $b_1,\dots,b_k,a_1,\dots,a_k$ witness $\sind(\overline{G})\geq k=\sind(G)$.
\end{proof}

Next, we note that the strong index is functionally equivalent to the largest order of a universal threshold graph that can be found as an induced subgraph.

\begin{fact}	\label{fact:sind_induced}
	Let $G$ be a graph.
	The maximum integer $k$ such that $G$ contains an induced $R_k$ is between $\lfloor\sind(G)/2\rfloor$ or $\sind(G)$.
	Consequently, if a graph excludes a threshold graph with $k$ vertices as an induced subgraph, then its strong index is smaller than $2k$.
\end{fact}
\begin{proof}
	If $G$ contains an induced $R_k$, then $\sind(G)\geq\sind(R_k)=k$. Conversely, assume $a_1,\dots,a_k,b_1,\dots,b_k$ witness $\sind(G)=k$. Then, $a_1,a_3,\dots,a_{2\lfloor k/2\rfloor-1}, b_2,\dots, b_{2\lfloor k/2\rfloor}$ induce 
	a graph isomorphic to $R_{\lfloor k/2\rfloor}$. The last claim follows immediately from~\cref{fact:univ_thr}.
\end{proof}

With the notion of strong index understood, we state and prove two lemmas that shall provide the main combinatorial leverage in the induction. Roughly speaking, the idea is that a {\gtd} of a graph and its complement can be used to decompose the graph into pieces that are simpler in terms of the strong index.

\begin{lemma}
	\label{lem:cog1}
	Let $G$ be a connected  graph and let $Y$ be a {\gtd} of $G$.
	Let $B$ be a bag of $Y$ with level $i>1$.
	Then $\sind(G[B])<\sind(G)$.
\end{lemma}
\begin{proof}
	By \cref{enum:hook} of
	\cref{def:dec}, the hook $h(B)$ is adjacent to all the vertices of $B$. Let $B'$ be the parent bag of $B$ in $Y$. Note that since $i>1$, $B'$ is not the root bag. By \cref{enum:hook} of
	\cref{def:dec} again, the hook $h(B')$ is non-adjacent to all the vertices of $B$. Let $a_1,\dots,a_k,b_1,\dots,b_k$ witness $\sind(G[B])=k$.
	Then $h(B),a_1,\dots,a_k,h(B'),b_1,\dots,b_k$ witness $\sind(G)>k$.
\end{proof}

\begin{lemma}
	\label{lem:cog2}
	Let $G$ be a connected graph, $Y$ be a {\gtd} of $G$, and
	$B$ be a bag of $Y$ with level $1$.
	Let $C$ be a connected component of $\ov{G[B]}$, $Y_C$ be a {\gtd} of $C$, and $B'$ be a bag of $Y_C$ with level $j\geq 1$. Then $\sind(G[B'])<\sind(G)$.
\end{lemma}
\begin{proof}
	For bags $B'$ of $Y_C$ with level $j>1$, the result follows from \cref{lem:cog1,fact:sindc}. So assume $j=1$, that is, $B'$ is a bag of $Y_C$ with level $1$.
	
	Let $r$ be the initial vertex of $Y_C$ and $r'$ be the initial vertex of $Y_C$. Note that in $G$, $r$ is adjacent to all the vertices in $B$ and $B'\subseteq B$, while $r'$ is non-adjacent to all the vertices of~$B'$. Therefore, if $a_1,\dots,a_k,b_1,\dots,b_k$ witness $\sind({G}[B'])=k$, then
	$r,a_1,\dots,a_k,r',b_1,\dots,b_k$ witness
	$\sind(G)> k=\sind(G[B'])$. 
\end{proof}

With all the tools prepared, we can proceed to the main argument. In essence, it boils down to combining \cref{lem:cog1,lem:cog2} in an induction on the strong index.

\begin{proof}[Proof of \cref{thm:induced}]
Fix $t$ for the remainder of the proof; we may assume that $t\geq 4$.
We shall prove that if a graph $G$ excludes $P_t$ and $\ov{P}_t$ as induced subgraphs and has strong index at most $k$, then $G$ has a cosplit of size at most 
\[
N_k\coloneqq \Bigl(3+\frac1{t-3}\Bigr)\cdot(2t-5)^{k-1}-\Bigl(\frac{1}{t-3}\Bigr)\leq 4\cdot(2t-5)^{k-1}.
\]
and height at most~$2k$. By \cref{fact:sind_induced}, this suffices to infer \cref{thm:induced}. We apply induction on $k$.

For the base case $k=1$,
consider any vertex $u\in V(G)$, let $X$ be the set of neighbors of $u$, and let $Y=V(G)\setminus X\setminus \{u\}$ be the set of non-neighbors of $u$. As $G$ has strong index $1$, it follows that the maximum degree in $G[Y]$ is at most $1$, that is, $G[Y]$ is a disjoint union of isolated vertices and isolated edges. Similarly, the maximum degree in $\ov{G}[X]$ is at most $1$, hence $\ov{G}[X]$ is a disjoint union of isolated vertices and isolated edges. It follows that $\{X,Y,\{u\}\}$ is a cosplit of $G$ of size at most $N_1=3$ and height at most $2$.

We proceed to the induction step. Let $Y$ be a {\gtd} of $G$. (Note that $G$ may be disconnected, so $Y$ may be a rooted forest.)
By \cref{lem:cog1}, for every bag $B$ of $Y$ with level $i>1$, we have $\sind(G[B])<\sind(G)$. Hence, by induction there is a cosplit $\Pp^B$ of $G[B]$ of size at most $N_{k-1}$ and height at most $2k-2$. Next, for every bag $B$ of $Y$ with level~$1$ and every connected component $C$ of $\ov{G}[B]$, we fix any {\gtd} $Y_C$ of~$C$. By \cref{lem:cog2}, for every bag $B'$ of $Y_C$ of level $j\geq 1$ we have $\sind(G[B'])<\sind(G)$. Therefore, by induction we again find a cosplit $\Pp^{B'}$ of $G[B']$ of size at most $N_{k-1}$ and height at most~$2k-2$. 

For convenience, let $\Ee$ be the set of all connected components $C$ of all graphs $\ov{G}[B]$, where $B$ ranges over bags of $Y$ with level $1$.
Note that by \cref{fact:height_gtd}, each of the \gtd s~$Y$ and $\{Y_C\colon C\in \Dd\}$ considered above has height at most $t-2$.


Every cosplit $\Pp^X$ defined in the paragraph above has size at most $N_{k-1}$, hence let us arbitrarily enumerate it as
$$\Pp^X = \{P^X_1,\ldots,P^X_{N_{k-1}}\},$$
where by abuse of notation some of the sets $P^X_\ell$ for $1\leq \ell\leq N_{k-1}$ may be empty. We define a partition $\Pp$ of $V(G)$ as follows:
\begin{itemize}
 \item For every $2\leq i\leq t-2$ and $1\leq \ell\leq N_{k-1}$, add to $\Pp$ the set 
 $$P_{i,\ell}\coloneqq \bigcup \left\{P^B_\ell\colon B\textrm{ is a bag of level }i\textrm{ in }Y\right\}.$$
 \item For every $1\leq j\leq t-2$ and $1\leq \ell\leq N_{k-1}$, add to $\Pp$ the set
 $$P_{1,j,\ell}\coloneqq \bigcup \left\{ P^{B'}_\ell\colon C\in \Ee\textrm{ and } B'\textrm{ is a bag with level }j\textrm{ in }Y_C\right\}.$$
 \item Finally, add to $\Pp$ the sets
 \begin{gather*}
   P_0 \coloneqq \bigcup \{B\colon B\textrm{ is a bag of }Y\textrm{ with level }0\},\\
   P_{1,0}\coloneqq \bigcup \{B'\colon C\in \Ee \textrm{ and }B'\textrm{ is a bag of }Y_C\textrm{ with level }0\}.
 \end{gather*}
\end{itemize}
Clearly, we have
$$|\Pp|\leq (t-3)\cdot N_{k-1}+(t-2)\cdot N_{k-1}+2=N_k.$$ 

It remains to verify that every part of $\Pp$ induces a cograph of height at most $2k$. For parts of the form $P_{i,\ell}$, $G[P_{i,\ell}]$ is the disjoint union of cographs $G[P^B_\ell]$ for $B$ ranging as in the definition. For parts of the form $P_{1,j,\ell}$, $G[P_{1,j,\ell}]$ can be obtained from cographs $G[P^{B'}_\ell]$ for $B'$ ranging as in the definition as follows: first, for every bag $B$ of $Y$ with level $1$ construct the join of the graphs $G[P^{B'}_\ell]$ for all $B'$ satisfying $B'\subseteq B$, and then take the disjoint union of the obtained graphs for all $B$ as above. Finally, $G[P_0]$ is edgeless and $G[P_{1,0}]$ is a disjoint union of cliques.
\end{proof}

%


%

\section{Constructing a $2$-cosplit: proof of \cref{thm:semi-induced}}

%

Our first step towards the proof of \cref{thm:semi-induced} is to reduce the problem to the case of bipartite graphs. For this, we adjust the notion of a $2$-cosplit to bipartite graphs as follows: if $G$ is a bipartite graph and $\Pp$ is a $2$-cosplit of $G$, then we require that $\Pp$ refines the bipartition; that is, every part of $\Pp$ is entirely contained in one of the sides. Note that thus, the only (semi-)induced subgraphs considered in the definition that may not be edgeless are the graphs $G[A,B]$ where $A,B$ are parts contained in the opposite sides. Further, when talking about induced subgraphs of bipartite graphs, we also treat them as bipartite graphs with the inherited bipartition.

The statement for bipartite graphs is provided below.

\begin{theorem}\label{thm:bipartite-depth2}
	For every pair of integers $t,k\in \N$ there exists $N\in \N$ such that the following holds: Every bipartite graph that excludes $P_t$, $\ti{P}_t$, and $H_k$ as induced subgraphs admits a depth-$2$ cosplit of size at most $N$ and height at most $4k$.
\end{theorem}

We remark that our proof of the above gives $N=3^{t^{2k-2}}$.
Before we give a proof of \cref{thm:bipartite-depth2}, let us see how \cref{thm:semi-induced} can be derived by combining it with \cref{thm:induced}.

\begin{proof}[Proof of \cref{thm:semi-induced} assuming \cref{thm:bipartite-depth2}]
 Since $G$ excludes $P_t$, $\ti{P}_t$, and $H_k$ as semi-induced subgraphs, it follows that $G$ excludes $P_t$, $\ov{P}_t$, and $R_k$ as induced subgraphs. So by \cref{thm:induced} there is a cosplit $\Pp_0$ of $G$ of size at most $N_0$ and depth at most $4k$, where $N_0$ depends only on~$t$ and~$k$. Next, for every pair of distinct parts $X,Y\in \Pp_0$ consider the semi-induced subgraph $G[X,Y]$. This is a bipartite graph that excludes $P_t$, $\ti{P}_t$, and $H_k$ as induced subgraphs, hence by \cref{thm:bipartite-depth2}, $G[X,Y]$ admits a $2$-cosplit $\Pp_{X,Y}$ of size at most $N_1$ and height at most $4k$, where $N_1$ depends only on $t$ and $k$. Let $\Pp$ be the coarsest partition of the vertex set of $G$ that refines all the partitions $\Pp_0$ and $\Pp_{X,Y}$ for distinct $X,Y\in \Pp_0$ in the following sense: for every $A\in\Pp$ and $B$ belonging to any of the partitions above, we have $A\cap B=\emptyset$ or $A\subseteq B$. Then $|\Pp|\leq N_0\cdot N_1^{N_0-1}\eqqcolon N$ and it is easy to argue, using \cref{fc:co-bi-co}, that $\Pp$ is a $2$-cosplit of $G$ of height at most~$4k$.
\end{proof}

%

Therefore, we are left with proving \cref{thm:bipartite-depth2}. As we deal with bipartite graphs, it will be convenient to introduce a variant of the strong index invariant that is suited for this setting.

\begin{definition}\label{def:bindex}
	The \emph{bipartite index} of a bipartite graph $G$, denoted $\bind(G)$, is the
	maximum integer $k$ such that $G$ contains vertices
	$a_1,\dots,a_k,b_1,\dots,b_k$ satisfying the following: all the $a_i$'s belong to one side of $G$, all the $b_j$'s belong to the other side of $G$, and 
	for all $1\leq i<j\leq k$
	the vertex $a_i$ is adjacent to the vertex $b_j$ and the vertex
	$b_i$ is not adjacent to the vertex $a_j$.
\end{definition}

We remark that the same concept was used under the name {\em{quasi-index}} in~\cite{StableTWW}. We prefer to use the term {\em{bipartite index}} here, as it is more descriptive. Note the following analogue of \cref{fact:sindc}.

\begin{fact}
	For every bipartite graph $G$ we have $\bind(G)=\bind(\ti{G})$.
\end{fact}

Next, let us characterize connected bipartite graphs of bipartite index $1$.

\begin{fact}
	\label{fact:index1}
	If $G$ is bipartite, connected, and $\bind(G)=1$, then $G$ is a complete bipartite graph.	
\end{fact}
\begin{proof}
    Let $L,R$ be the sides of $G$. By contradiction, suppose there exist $u\in L$ and $v\in R$ that are not adjacent. Observe that every vertex $u'\in L\setminus \{u\}$ must be non-adjacent to every vertex $v'\in R\setminus \{v\}$, for otherwise $u',u,v,v'$ would witness $\bind(G)\geq 2$. So there is no edge with one endpoint in $L\setminus\{u\}\cup\{v\}$ and second in $R\setminus\{v\}\cup\{u\}$; this contradicts the assumption that~$G$ is connected.
\end{proof}

Also, we have the following counterpart of \cref{fact:sind_induced}.

\begin{fact}
	\label{lem:Hk-ind}
	If a bipartite graph has bipartite index at least $2k+1$, then it contains
	an induced~$H_k$.
\end{fact}
\begin{proof}
	Suppose $a_1,\dots,a_{2k+1},b_1,\dots,b_{2k+1}$ witness
	$\bind(G)\geq 2k+1$. 
	Color every integer $i\in\{1,\ldots,2k+1\}$ black or white depending on whether $a_i$ is adjacent to $b_i$ or not.
	Then there exists a monochromatic subset of
	indices of size $k+1$. The vertices with those indices induce either
	$H_{k+1}$ or $\ti{H}_{k+1}$ in $G$. We conclude by observing that both $H_{k+1}$ and $\ti{H}_{k+1}$ contain
	$H_k$ as an induced subgraph.
\end{proof}
%
%
%
%
%
%

Finally, the following three simple lemmas will be useful for merging $2$-cosplits.

\begin{lemma}\label{lem:merge}
	Let $G$ be a graph. If every connected component of $G$ has a $2$-cosplit of size at most $N$ and height at most $h$, then $G$ has a $2$-cosplit of size at most $N$ and height at most $h+1$.
\end{lemma}
\begin{proof}
 For a connected component $C$ of $G$, let $\Pp^C$ be a $2$-cosplit of $C$ of size at most $N$ and height at most $h$. Arbitrarily enumerate $\Pp^C$ as $\{P^C_1,\ldots,P^C_N\}$, where by abuse of notation some of the parts may be empty. Then, for $j\in \{1,\ldots,N\}$, define
 $$P_j\coloneqq \bigcup \{P^C_j\colon C\textrm{ is a connected component of }G\}.$$
 It is straightforward to verify that $\{P_1,\ldots,P_N\}$ is a $2$-cosplit of $G$ of height at most $h+1$.
\end{proof}

\begin{lemma}\label{lem:partition}
 Let $G$ be a bipartite graph and let $\Rr$ be a partition of $V(G)$ such that every part of $\Rr$ is entirely contained in one side of $G$. Suppose for every pair of parts $A,B\in \Rr$ belonging to opposite sides of $G$, the graph $G[A,B]$ admits a $2$-cosplit of size at most $N$ and height at most $h$.
 Then $G$ admits a $2$-cosplit of size at most $|\Rr|\cdot N^{|\Rr|}$ and height at most $h$.
\end{lemma}
\begin{proof}
 For each pair of parts $A,B\in \Rr$ belonging to opposite sides of $G$, let $\Pp_{A,B}$ be the assumed $2$-cosplit of $G[A,B]$ of size at most $N$ and height at most $h$. Let $\Pp$ be the coarsest partition of $V(G)$ that refines $\Rr$ and all the $2$-cosplits $\Pp_{A,B}$ in the following sense: for all pairs $A,B\in \Rr$ as above and all parts $D\in \Pp_{A,B}$, every part of $\Pp$ is either contained in or disjoint with $D$. Then $|\Pp|\leq |\Rr|\cdot N^{|\Rr|}$, because every part of $\Rr$ contains at most $N^{|\Rr|}$ different parts of $\Pp$, and it is straightforward to verify that $\Pp$ is a $2$-cosplit of $G$ of height at most $h$.
\end{proof}

\begin{lemma}\label{lem:compl}
 Every $2$-cosplit of a bipartite graph $G$ is also a $2$-cosplit of $\ti{G}$ of the same height.
\end{lemma}
\begin{proof}
 Follows from the fact that bi-cographs are closed under bipartite complementation. Note that this operation preserves the height.
\end{proof}

Next, we prove the analogues of \cref{lem:cog1,lem:cog2}. In both cases the reasoning follows the same path, but one needs to be careful about the bipartiteness.
%

\begin{lemma}
	\label{lem:bag_ind}
	Let $G$ be a connected bipartite graph and let $Y$ be a {\gtd} of $G$.
	Let $B$ be a bag of $Y$ with level $i>1$ and let $G_B$ be the
	subgraph of $G$ induced by $B$ and the union of all the descendants of $B$ at levels $i'\not\equiv i\bmod 2$.
	Then $\bind(G_B)<\bind(G)$.
\end{lemma}
\begin{proof}
    Let $D$ be the union of all the descendants of $B$ at levels $i'\not\equiv i\bmod 2$.
    By \cref{lem:bip-gtd}, each of $B$ and $D$ is entirely contained in a single side of $G$, and these are different sides.
    
	By \cref{enum:hook} of
	\cref{def:dec}, the hook $h(B)$ is adjacent to all the vertices of
	$B$. Let $B'$ be the parent of $B$ in $Y$. Note that since $i>1$, $B'$ is not the root of $Y$. By \cref{enum:hook} of
	\cref{def:dec} again, the hook $h(B')$ is non-adjacent to all the vertices of $G_B$. Moreover, by \cref{lem:bip-gtd}, $h(B')$ belongs to the same side as $B$ and $h(B)$ belongs to the same side as $D$.  Let
	$a_1,\dots,a_p,b_1,\dots,b_p$ witness $\ind(G_B)=p$.  By
	exchanging $a_\ell$ with $b_{p+1-\ell}$ (for every $\ell\in \{1,\ldots,p\}$) if necessary we may
	assume that $a_1,\ldots,a_p\in D$ and $b_1,\dots,b_p\in B$. Then $h(B),a_1,\dots,a_p,h(B'),b_1,\dots,b_p$ witness
	$\bind(G)\geq p+1$. It follows that $\bind(G_B)<\bind(G)$. 
\end{proof}
%

\begin{lemma}
	\label{lem:bag_indc}
	Let $G$ be a connected bipartite graph and let $Y$ be a {\gtd} of $G$.
	Let $B$ be a bag of $Y$ with level $1$ and let $G_B$ be the
	subgraph of $G$ induced by $B$ and the union of all the bags that are descendants of $B$  in $Y$ at levels $i'\not\equiv 1\bmod 2$.
	Let $C$ be a connected component of $\ti{G}_B$ and $Y_C$ be a {\gtd} of $H$ rooted at a vertex $r_C\in V(C)\cap B$.
	Let $B'$ be a bag of $Y_C$ with level $j\geq 1$ and let $G_{B'}$ be the
	subgraph of $G$ induced by $B'$ and the union of all the bags of $Y_C$ that are descendants of $B'$ in $Y_C$ at levels $j'\not\equiv j\bmod 2$.
	Then $\bind(G_{B'})<\bind(G)$.
\end{lemma}
\begin{proof}
	For bags $B'$ with level $j>1$ the result follows from \cref{lem:bag_ind}. So let $B'$ be a bag of $Y_C$ with level $1$. Let $L,R$ be the sides of $G$, where the initial vertex $r$ of $Y$ belongs to $L$. Then $B\subseteq R$, so in particular $r_C\in R$. Letting $D'=V(G_{B'})\setminus B'$ be the union of descendants of $B'$ in $Y_C$ at levels $j'\not\equiv 1\bmod 2$, we have $B'\subseteq L$ and $D'\subseteq R$.
	
	Let $a_1,\dots,a_p,b_1,\dots,b_p$ witness $\bind(G_{B'})=p$.  By
	exchanging $a_\ell$ with $b_{p+1-\ell}$ (for every $\ell\in \{1,\ldots,p\}$) if necessary we may
	assume that $a_1,\dots,a_p\in B'$ and $b_1,\ldots,b_p\in D'$. Note that $D'=V(G_{B'})\cap R\subseteq V(G_B)\cap R=B$, so $b_1,\dots,b_p\in B$.
	As $r$ adjacent in $G$ to all the vertices in~$B$ and as $r_C$ is non-adjacent in $G$ to all the vertices in $B'$, the vertices $r,a_1,\dots,a_p,r',b_1,\dots,b_p$ witness
	$\bind(G)\geq p+1$. It follows that $\bind(G)>\bind(G_{B'})$. 
\end{proof}

With all the tools prepared, we are ready to complete the argument.


\begin{proof}[Proof of \cref{thm:bipartite-depth2}]
Fix $t$ for the remainder of the proof; we may assume without loss of generality that $t\geq 5$.
	For $k\geq 1$, let 
	\begin{align*}
	N_k&=\frac{\bigl(2t^{1/(t-1)}\bigr)^{\mathrlap{t^{2k-2}}}}{t^{1/(t-1)}}.
\intertext{
	Note that $N_1=2$ and that the sequence satisfies the recurrence }
	    N_k &=  t^{t+1}\cdot N_{k-1}^{t^2}&\text{(for $k\geq 2$)}.
	\end{align*}
	
For a given $k$,
	let $Q(k)$ be the statement that every bipartite graph that excludes $P_t$ and $\ti{P}_t$ as induced subgraphs and has bipartite index at most $k$, admits a cosplit of size at most $N_k$ and height at most $2k$.  
	We shall prove $Q(k)$ by induction on $k$. Provided we achieve this, \cref{thm:bipartite-depth2} will immediately follow by \cref{lem:Hk-ind}. 
	
	
	For the base case $k=1$, \cref{fact:index1} implies that a bipartite graph of bipartite index $1$ is a disjoint union of bicliques. Hence it admits a trivial $2$-cosplit of size at most $N_1=2$ and height at most $2$.
	
	We proceed to the induction step; that is, we need to prove $Q(k)$ assuming $Q(k-1)$ for $k\geq 2$. Let $L,R$ be the bipartition of $G$.
	By performing the reasoning in $\ti{G}$ instead of in $G$ and using \cref{lem:compl} if necessary, we may assume that at least one of the sides $L$ and $R$ has no isolated vertices. Further, by switching the sides if necessary, we may assume that there are no isolated vertices in $R$. Hence, we may choose a {\gtd} $Y$ of $G$ so that every root bag of $Y$ is contained in~$L$. By \cref{fact:height_gtd}, $Y$ has height at most $t-2$.
	
	For each $i\in \{0,1,\ldots,t-2\}$, let $U_i$ be the union of bags of $Y$ with level $i$. Thus $U_i\subseteq L$ when~$i$ is even and $U_i\subseteq R$ when $i$ odd. Further, let 
	$$G_i\coloneqq G[U_i,U_{i+1}\cup U_{i+3}\cup \ldots \cup U_{t'-2}\cup U_{t'}],$$
	where $t'\in \{t-3,t-2\}$ is of different parity than $i$. We argue the following claim:
	\begin{equation}\label{eq:cosplit}\tag{$\clubsuit$}
	 \parbox{\dimexpr\linewidth-4em}{For each $i\in \{0,1,\ldots,t-2\}$, the graph $G_i$ admits a $2$-cosplit of size at most $t\cdot N_{k-1}^{t}$ and height at most $2k$.} 
	\end{equation}
	Observe that $Q(k)$ follows from~\eqref{eq:cosplit} combined with \cref{lem:partition}; here we use the partition $\Rr=\{U_0,U_1,\ldots,U_{t-2}\}$. Hence, we are left with proving~\eqref{eq:cosplit}.
	
	For every bag $B$ of $Y$, let $G_B$ be the subgraph of $G$ induced by $B$ and the union of all descendants of $B$ at levels of parity different from that of the level of $B$. Observe that for each $i\in \{0,1,\ldots,t-2\}$, $G_i$ is the disjoint union of graphs $G_B$ for $B$ ranging over bags with level $i$. Therefore, from \cref{lem:merge} we conclude that in order to prove~\eqref{eq:cosplit} for $i$, it suffices to prove the following claim:
	\begin{equation}\label{eq:cosplit-single}\tag{$\diamondsuit$}
	 \parbox{\dimexpr\linewidth-4em}{For each bag $B$ of $Y$ with level $i$, the graph $G_B$ admits a $2$-cosplit of size at most $t\cdot N_{k-1}^{t}$ and height at most $2k-1$.} 
	\end{equation}
	Hence, from now on we focus on proving~\eqref{eq:cosplit-single}.
	
	First, consider the case $i=0$. Then $B=\{u\}$ for some vertex $u$, and all other vertices of $G_B$ are on the opposite side of $u$. Therefore, $G_B$ is the disjoint union of a star and a collection of isolated vertices, so $G_B$ admits a trivial $2$-cosplit of size $2$ and depth $2$. This settles~\eqref{eq:cosplit-single} for $i=0$.
    
    Next, consider the case $i>1$. By \cref{lem:bag_ind}, $\bind(G_B)<\bind(G)$ for each $B$ as above, hence by induction $G_B$ admits a cosplit of size at most $N_{k-1}$ and height at most $2k-2$. This settles~\eqref{eq:cosplit-single} for $i>1$.
    
    We are left with the case $i=1$. Let $G'_B$ be the graph obtained from $G_B$ by removing all vertices of $V(G_B)\cap L$ that are isolated in $\ti{G}_B$. Further, let $Y'$ be a \gtd{} of $\ti{G}_B'$. Since no vertex contained $L$ is isolated in $\ti{G}_B'$, we can choose $Y'$ so that all root bags of $Y'$ are contained in~$R$ (and thus also in $B$, because $R\cap V(G_B)=B$). By \cref{fact:height_gtd}, $Y'$ has height at most $t-2$.
    
    For $j\in \{0,1,\ldots,t-2\}$, let $W_j$ be the union of bags of $Y'$ with level $j$, where in $W_1$ we additionally include all vertices of $V(G_B)\setminus V(G'_B)$. Note that thus $\{W_0,\ldots,W_{t-2}\}$ is a partition of $V(G_B)$, $W_j\subseteq R$ for even $j$ and $W_j\subseteq L$ for odd $j$. 
    Denote 
	$$G'_{j}\coloneqq G[W_j,W_{j+1}\cup W_{j+3}\cup \ldots \cup W_{t'-2}\cup W_{t'}],$$
	where $t'\in \{t-3,t-2\}$ is of different parity than $j$.
    By \cref{lem:partition}, to prove \eqref{eq:cosplit-single} it suffices to prove the following.
    \begin{equation}\label{eq:cosplit-co}\tag{$\heartsuit$}
	 \parbox{\dimexpr\linewidth-4em}{For every $j\in \{0,1,\ldots,t-2\}$, the graph $G'_j$ admits a $2$-cosplit of size at most $N_{k-1}$ and height at most $2k-1$.} 
	\end{equation}
	Hence, from now on we focus on proving~\eqref{eq:cosplit-co}.
	
	For every bag $B'$ of $Y'$, let $G'_{B'}$ be the subgraph of $G$ induced by $B'$ and all descendants of $B'$ in~$Y'$ with levels of parity different from the level of $B'$. Observe that for every $j\in \{0,1,\ldots,t-2\}$, the graph $G'_j$ is the bipartite join of graphs $G'_{B'}$ for $B'$ ranging over bags of $Y'$ with level $j$. (In the case $j=1$ we also need to include in the join the single-vertex graph on $u$ for every $u\in V(G_B)\setminus V(G_B')$.) Therefore, by \cref{lem:merge,lem:compl}, to prove~\eqref{eq:cosplit-co} it suffices to prove the following.
	\begin{equation}\label{eq:cosplit-coco}\tag{$\spadesuit$}
	 \parbox{\dimexpr\linewidth-4em}{For every bag $B'$ of $Y'$ with level $j$, the graph $G'_{B'}$ admits a $2$-cosplit of size at most~$N_{k-1}$ and height at most $2k-2$.} 
	\end{equation}
	
	In the case $j=0$, we again have that $B'=\{u\}$ for some vertex $u$ and all other vertices of $G'_{B'}$ belong to the opposite side of this bipartite graph. It follows that $G'_{B'}$ is the disjoint union of a star and a collection of isolated vertices, so $G'_{B'}$ admits a trivial $2$-cosplit of size $2$ and height $2$. This settles~\eqref{eq:cosplit-coco} for $j=0$.
	
	We are left with the final case $j\geq 1$. By \cref{lem:bag_indc}, we have $\bind(G'_{B'})<\bind(G)$. So by induction we infer that $G'_{B'}$ admits a $2$-cosplit of size at most $N_{k-1}$ and height at most $2k-2$. This settles~\eqref{eq:cosplit-coco} for $j\geq 1$ and concludes the proof.
\end{proof}

\section{Sparsification: proof of \cref{thm:main-shb}}

In this section we use the results gathered so far to prove our main result, \cref{thm:main-shb}. As mentioned, the strategy is to show that for fixed $t,k\in \N$, graphs excluding semi-induced~$P_t$,~$\ti{P}_t$, and $H_k$ can be sparsified using an $\FO$ transduction. Then we exploit the known connections between treedepth and the existence of long paths as subgraphs.

We first show that cographs with bounded height can be efficiently sparsified. Here, by a {\em{coloring procedure}} we mean an algorithm that inputs a graph and outputs some coloring of it\footnote{Formally, we defined colorings to use an infinite number of colors, but in what follows, coloring procedures always output colorings using a finite number of colors. One may imagine that the remaining colors are set to be empty sets and are thus omitted in the description of the output. They are also never used by any formula that we are about to construct.}.

\begin{lemma}
	\label{lem:sparsify_cograph}
	For every  positive integer $h$ there exist a  polynomial-time coloring procedure $\mathsf A$, a class $\mathscr D_1$ of colored graphs with treedepth at most $h$, and two first-order formulas $\varphi_1,\psi_1$, such that 
	$(\varphi_1,\psi_1)$ is a simple equivalence of the classes $\mathsf A(\mathscr C)$ and $\mathscr D_1$, where $\mathscr C$ is the class of all cographs with height at most $h$.
\end{lemma}
\begin{proof}
We first define a sequence of partitions $\Pp_0,\ldots,\Pp_h$ of the vertex set of $G$ as follows:
\begin{itemize}
\setlength\itemsep{-.1em}
 \item $\Pp_0=\{V(G)\}$.
 \item For odd $i\geq 1$, $\Pp_i$ is obtained from $\Pp_{i-1}$ by splitting every part $A\in \Pp_{i-1}$ into the vertex sets of the connected components of $G[A]$.
 \item For even $i\geq 1$, $\Pp_i$ is obtained from $\Pp_{i-1}$ by splitting every part $A\in \Pp_{i-1}$ into the vertex sets of the connected components of $\ov{G}[A]$. 
\end{itemize}

It is straightforward to see that provided $G$ is a cograph of height at most $h$, $\Pp_h$ will be the discrete partition with every vertex in its own part. Clearly, partitions $\Pp_0,\ldots,\Pp_h$ can be computed in polynomial time.

Next, for $0\leq i\leq h$ we inductively define formulas $\varpi_i(x,y)$ expressing that $x$ and $y$ are in the same part of $\Pp_i$. For $i=0$ this is trivial. For odd $i\geq 1$, $\varpi_i(x,y)$ can be obtained from $\varpi_{i-1}(x,y)$ by checking that $\varpi_{i-1}(x,y)$ holds and $x$ and $y$ are in the same connected component of the subgraph of $G$ induced by the part of $\Pp_{i-1}$ to which $x$ and $y$ belong. Note here that in the setting of of cographs, belonging to the same connected component is equivalent to being at distance at most $3$, hence it can be expressed in $\FO$. The construction for even $i\geq 1$ is the same, except we complement the subgraph induced by the part of $x$ and $y$.

The coloring procedure $\mathsf A$ first computes the partitions $\Pp_0,\ldots,\Pp_h$. Then, for each $i\in \{0,1,\ldots,h\}$, it marks one arbitrary vertex in each part of $\Pp_i$ with predicate $P_i$. The coloring of~$G$ obtained in this way is the output of $\mathsf A$ on $G$.

%


The formula $\varphi_1(x,y)$ expresses that for some $0\leq i\leq h$, $\varpi_i(x,y)$ holds and one of $x$ and $y$ is marked with $P_i$. It is straightforward to see that then $\mathsf I_{\varphi_1}(\mathsf A(G))$ is a colored graph with treedepth at most $h$.

Denote by $N[v]$ the closed neighborhood of a vertex $v$ (i.e. the set formed by $v$ and all its neighbors).
It is easily checked that two vertices $u$ and $v$ of $G$ are adjacent if and only if, in $\mathsf I_{\varphi_1}(\mathsf A(G))$, the maximum integer $i$
such that $N[u]\cap N[v]$ contains a vertex marked with $P_i$ is odd. As $0\leq i\leq h$, this property is first-order definable by a formula $\psi_1$.
\end{proof}

We now give an analogous result for bi-cographs with bounded height. 
As the proof is very similar (considering bi-complement instead of complements), we omit it.

\begin{lemma}
		\label{lem:sparsify_bi-cograph}
	For every  positive integer $h$ there exist a  polynomial time coloring procedure $\mathsf B$, a class $\mathscr D_2$ of colored graphs with treedepth at most $h$, and two first-order formulas $\varphi,\psi$, such that 
	$(\varphi_2,\psi_2)$ is a simple equivalence of the classes $\mathsf B(\mathscr B)$ and $\mathscr D_2$, where $\mathscr B$ is the class of all bi-cographs with height at most $h$.
\end{lemma}

The next theorem is an easy consequence of the two preceding lemmas.

\begin{theorem}
	\label{thm:sparsification_cosplit}
	Let $N,k$ be integers, and let $\mathscr C$ be the class of all graphs that admit a $2$-cosplit of size at most $N$ and height at most $k$. Then there exist a coloring procedure $\mathsf C$, a class $\mathscr D$ of $dN$-degenerate graphs, and two first-order formulas $\varphi,\psi$, such that the following hold:
\begin{itemize}
	\item Given a graph $G$ and a $2$-cosplit of $G$ of size at most $N$ and height at most $k$, the procedure $\mathsf C$ computes the coloring $\mathsf C(G)$ in polynomial time.
	\item The pair $(\varphi,\psi)$ is a simple equivalence of the classes $\mathsf C(\mathscr C)$ and $\mathscr D$.
\end{itemize}
\end{theorem}
\begin{proof}
	The coloring procedure $\mathsf C$ first colors the vertices according to the $2$-cosplit.
	Then, it calls the procedure $\mathsf A$ of \cref{lem:sparsify_cograph} on the graph induced by each part, and the procedure $\mathsf B$ of \cref{lem:sparsify_bi-cograph} on the bipartite subgraph induced by each pair of distinct parts. Here, all unary predicates used by all invocation of procedures $\mathsf A$ and $\mathsf B$ are kept distinct. Then, one can easily construct the formulas $\varphi$ and $\psi$ using a disjunction on the possible parts of the $2$-cosplit to which $x$ and $y$ belong.
\end{proof}

By combining \cref{thm:sparsification_cosplit} with \cref{thm:semi-induced} we derive the following statement.

\begin{corollary}
	\label{thm:sparsification}
	Every class of graphs that excludes a path, the bipartite complement of a path, and a half-graph as semi-induced subgraphs is transduction equivalent to a class of degenerate graphs.
\end{corollary}


We are now ready to prove our main result.

\begin{proof}[Proof of \cref{thm:main-shb}]
	Let $\Cc$ be a class of graphs such that the class of all paths cannot be 
	$\FO$-transduced from $\Cc$. Then there are integers $t,k\in \N$ such that no graph in $\Cc$ contains $P_t$, $\ti{P}_t$, or $H_k$ as a semi-induced subgraph.
	Therefore, by \cref{thm:sparsification}, $\Cc$ is transduction equivalent to a class $\Dd$ of degenerate graphs. Clearly, it is still the case that the class of all paths cannot be $\FO$-transduced from $\Dd$, hence graphs in $\Dd$ do not contain all paths as induced subgraphs. Since $\Dd$ is in addition degenerate, 
	by~\cite[Proposition 6.4]{Sparsity}
	 it  follows that $\Dd$ has bounded treedepth, hence also bounded shrubdepth. So by the results of~\cite{Ganian2017}, $\Cc$ has bounded shrubdepth as well.
\end{proof}

\paragraph*{Acknowledgements} The research leading to the results presented in this work was initiated during Dagstuhl Seminar 21391 {\em{Sparsity in Algorithms, Combinatorics, and Logic}}. We are grateful to the Dagstuhl staff for creating such a stimulating work atmosphere. We also thank multiple other participants of the workshop for inspiring discussions that greatly helped in the shaping of this~work.

\bibliographystyle{amsplain}
\bibliography{ref}
\end{document}